\documentclass[runningheads]{llncs}
\usepackage{amsmath,amssymb}

\usepackage{tikz}
\usepackage{pgf}
\usetikzlibrary{positioning,arrows,shapes,decorations.markings,decorations.pathreplacing,matrix}
\tikzstyle{vertex}=[circle,draw=black,fill=black,inner sep=0,minimum size=3pt,text=white,font=\footnotesize]

\begin{document}
	
\title{Coloring Hasse diagrams and disjointness graphs of curves\thanks{Research partially supported by Swiss National Science Foundation grants no. 200020-162884 and 200021-175977.}}

\author{J\'anos Pach \inst{1,2}\orcidID{0000-0001-8344-3592}\and
	Istv\'an Tomon \inst{1}\orcidID{0000-0002-2389-2035}
}

\institute{\'{E}cole Polytechnique F\'{e}d\'{e}rale de Lausanne, Lausanne, Switzerland \and R\'enyi Institute, Budapest, Hungary
	\email{\{janos.pach,istvan.tomon\}@epfl.ch}}

\maketitle

\begin{abstract}
	Given a family of curves $\mathcal{C}$ in the plane, its disjointness graph is the graph whose vertices correspond to the elements of $\mathcal{C}$, and two vertices are joined by an edge if and only if the corresponding sets are disjoint. We prove that for every positive integer $r$ and $n$, there exists a family of $n$ curves whose disjointness graph has girth $r$ and chromatic number $\Omega(\frac{1}{r}\log n)$. In the process we slightly improve Bollob\'as's old result on Hasse diagrams and show that our improved bound is best possible for uniquely generated partial orders.
	
\end{abstract}

\section{Introduction}

There are two important, seemingly unrelated, concepts that play important roles in Geometric Graph Theory and in Graph Drawing: {\em Hasse diagrams} and {\em string graphs}.

Hasse diagrams were introduced by Vogt~\cite{V95} at the end of the 19th century for concise representation of partial orders. Today they are widely used in graph drawing algorithms. Let $P$ be a partially ordered set with partial ordering $\prec$. For any $x,y\in P$, we say that $y$ \emph{covers} $x$ if $x\prec y$ and there is no $z\in P$ such that $x\prec z\prec y$. The \emph{Hasse diagram} of $P$ is the directed graph on the elements of $P$, where there is an edge from $x$ to $y$ if and only if $y$ covers $x$. If we disregard the direction of the edges, we obtain the {\em cover graph} of $P$. The graph on $P$ whose two elements are connected by an edge if and only if they are related by $\prec$ is the {\em comparability graph} of $P$. The cover graph is a subgraph of the comparability graph.

The \emph{intersection graph} of a family of sets $\mathcal{C}$ is the graph whose vertices correspond to the elements of $\mathcal{C}$ and two vertices are joined by an edge if and only if the corresponding sets have a nonempty intersection. The \emph{disjointness graph} of $\mathcal{C}$ is the complement of the intersection graph of $\mathcal{C}$. A {\em string}, or \emph{curve}, $\gamma$ is the image of a continuous function $f:[0,1]\rightarrow \mathbb{R}^{2}$. A curve $\gamma$ is \emph{grounded} if one of its endpoints is on the $y$-axis, and $\gamma$ lies in the nonnegative half-plane. See Figure \ref{figure1} for an illustration of a grounded family of curves and its disjointness graph.   A {\em string graph} is the intersection graph of curves. The notion was introduced by Benzer~\cite{B59} and Sinden~\cite{Si66} to describe the incidence structures of intervals in chromosomes and metallic layers in printed networks, respectively. The systematic study of string graphs was initiated in \cite{EET76} and \cite{Gr78}.
\smallskip

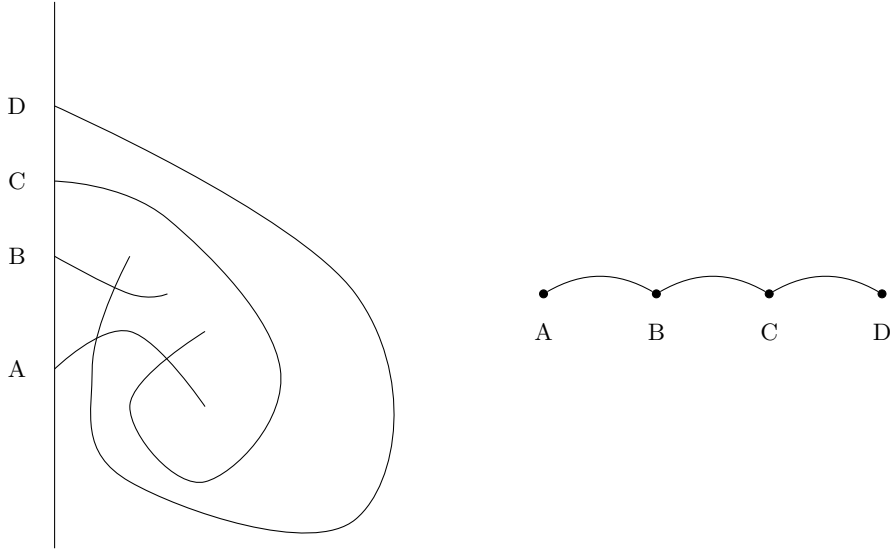
\begin{figure}
	\begin{center}
		\begin{tikzpicture}
		
		\node (v1) at (0,6) {};
		\node (v2) at (0,-1.5) {};
		\draw  (v1) edge (v2);
		\node at (-0.5,1) {A};
		\node at (-0.5,2.5) {B};
		\node at (-0.5,3.5) {C};
		\node at (-0.5,4.5) {D};
		\draw  plot[smooth, tension=.7] coordinates {(0,1) (1,1.5) (2,0.5)};
		\draw  plot[smooth, tension=.7] coordinates {(0,2.5) (1,2) (1.5,2)};
		\draw  plot[smooth, tension=.7] coordinates {(0,3.5) (1.5,3) (3,1) (2,-0.5) (1,0.5) (2,1.5)};
		\draw  plot[smooth, tension=.7] coordinates {(0,4.5) (4,2) (4,-1) (1,-0.5) (0.5,1) (1,2.5)};
		\node[vertex] (v3) at (6.5,2) {};
		\node[vertex] (v4) at (8,2) {};
		\node[vertex] (v5) at (9.5,2) {};
		\node[vertex] (v6) at (11,2) {};
		\draw  (v3) edge[bend left] (v4);
		\draw  (v4) edge[bend left] (v5);
		\draw  (v5) edge[bend left] (v6);
		\node  at (6.5,1.5) {A};
		\node  at (8,1.5) {B};
		\node  at (9.5,1.5) {C};
		\node  at (11,1.5) {D};
		\end{tikzpicture}
		\caption{A family of grounded curves and its disjointness graph.  }
		\label{figure1}
	\end{center}
\end{figure}

The first sign that the above concepts are intimately related was the following simple fact discovered by Golumbic, Rotem, Urrutia~\cite{GoRU83}, and Lov\'asz~\cite{Lo83}: Every comparability graph is the disjointness graph of a collection of curves in the plane. A partial converse of this statement was established in~\cite{FoP12}.

A useful characterization of cover graphs in terms of strings follows directly from Corollary 2.7 of Middendorf and Pfeiffer~\cite{MiP93} and Theorem 1 in~\cite{Si66}. See also~\cite{Kr91} and~\cite{RoW18} (page 2).

\begin{theorem}\label{character} \cite{MiP93}, \cite{Si66}  A triangle-free graph is a cover graph of a partially ordered set if and only if it is isomorphic to the disjointness graph of a family of grounded curves.
\end{theorem}

The {\em girth} of a graph $G$ is the length of the shortest cycle in $G$. Obviously, every triangle-free graph has girth at least {\em four}. According to a classical result of Erd\H os~\cite{Er59}, for every $r\ge 3$, there exist graphs with $n$ vertices and girth at least $r$ which have arbitrarily large chromatic numbers. Erd\H os's construction is probabilistic and does not posses any geometric structure.

For geometrically defined graphs, the situation is more complicated. The chromatic number of intersection graphs of axis-parallel rectangles~\cite{AsG60} or chords of a cycle~\cite{Gy85,KoK97,DaM19} and disjointness graphs of segments in the plane~\cite{PaT94,To00} can be bounded from above by a function of their clique numbers. In sharp contrast to this, Pawlik, Kozik, Krawczyk, Laso\'n, Micek, Trotter, and  Walczak~\cite{PaKK14} proved that there exist triangle-free intersection graphs of $n$ segments with chromatic number $\Omega(\log \log n)$, which disproved a longstanding conjecture of Erd\H{o}s. In~\cite{PaTT17}, triangle-free disjointness graphs of $n$ curves were constructed, with chromatic number $\Omega(\log n)$, cf.~\cite{MuWW18}. This construction is based on shift graphs, defined by Erd\H os and Hajnal~\cite{ErH64}. It appears to be difficult to extend this method to obtain disjointness graphs of curves with high girth and high chromatic number.

The aim of the present note is to construct such graphs.

\begin{theorem}\label{thm:main1}
For every positive integer $r$ and for every sufficiently large $n$, there exists a family of $n$ curves whose disjointness graph has girth at least $r$ and chromatic number at least $\Omega(\frac{1}{r}\log n)$.
\end{theorem}

This result does not remain true if we are allowed to use only {\em $x$-monotone} curves, that is, if every vertical line meets each curve in at most one point. In this case, the chromatic number of the cover graph is bounded from above by a constant~\cite{PaT94,PaT19}.

In view of Theorem \ref{character}, in order to prove Theorem~\ref{thm:main1}, it is sufficient to establish the following.

\begin{theorem}\label{thm:chi_cover}
For every positive integer $r$ and for every sufficiently large $n$, there exists a poset on $n$ vertices whose cover graph has girth at least $r$ and chromatic number $\Omega(\frac{1}{r}\log n)$.
\end{theorem}

The study of combinatorial properties of cover graphs (Hasse diagrams) is an extensive area of research in the theory of partial orders. Bollob\'as~\cite{Bo77} was the first to show the existence of partial orders (actually, lattices) whose cover graphs have arbitrarily large girth and chromatic number. Alternative constructions were found by Ne\v set\v ril {\em et al.}~\cite{NeR79,KrN91}. Bollob\'as's proof, which gives the best known asymptotic bound, builds on Erd\H{o}s's probabilistic construction~\cite{Er59} mentioned above. It shows that for a fixed girth $r$ and $n\rightarrow\infty$, the chromatic number of a cover graph with $n$ vertices can be as large as $\Omega(\frac{\log n}{\log\log n})$. Our Theorem~\ref{thm:chi_cover} improves on this bound.

It is possible that Theorem~\ref{thm:chi_cover} can be further improved. However, we can show that our bound is tight for an interesting family of cover graphs. A partially ordered set $P$ is called \emph{uniquely generated} if for every comparable pair of vertices $x\prec y$, there exists a unique sequence of vertices $x=v_{1}\prec\dots\prec v_{k}=y$ such that $v_{i+1}$ covers $v_{i}$ for $i=1,\dots,k-1$. Obviously, if there is no chain with $3$ elements in $P$, then $P$ is uniquely generated and its cover graph is bipartite.

\begin{theorem}\label{thm:chi_ug}
	(i) If $P$ is a uniquely generated poset on $n$ vertices, then the chromatic number of its cover graph is at most $\lfloor \log_{2} n\rfloor+1$.
	
	(ii) For every integer $r>3$ and for every sufficiently large $n$, there exists a uniquely generated poset on $n$ vertices whose cover graph has girth at least $r$ and chromatic number at least $\Omega(\frac{1}{r}\log n)$.
\end{theorem}

\section{Cover graphs with large chromatic number}

In this section, we prove Theorem \ref{thm:chi_ug}. Note that then Theorem \ref{thm:chi_cover} is an immediate consequence of part (ii) of Theorem \ref{thm:chi_ug}. We omit floors an ceilings for easier readability.

\bigskip

\textit{Proof of Theorem \ref{thm:chi_ug}, part (i).} Let $G$ be the cover graph of $P$, let $<_{P}$ be the partial ordering on $P$, and let $\prec$ be a linear extension of $<_{P}$. For any $x\in P$, let $C(x)$ denote the set of vertices of $P$ covered by $x$.
	
	We prove that the greedy coloring of $G$ with respect to $\prec$ uses at most $1+\lfloor \log_{2} n\rfloor$ colors. Let $v_{1}\prec\dots\prec v_{n}$ be the vertices of $G$. Color them  with the elements of $\mathbb{Z}^{+}$, as follows. For $i=1,\dots,n$, if $v_{1},\dots,v_{i-1}$ have already been colored, then color $v_{i}$ with the smallest positive integer $k$ that does not appear among the colors of $C(v_{i})$.
	
	For each vertex $v\in V(G)$, let $T(v)$ denote the set of vertices $u\in V(G)$ such that $u\leq_{P} v$. Note that, as $P$ is uniquely generated, the subgraph of $G$ induced by $T(v)$ is a tree. We claim that if $v$ received color $k$, then $|T(v)|\geq 2^{k-1}$. This clearly implies (i), because if the total number of colors used by our coloring is $K$, then we have $n\geq 2^{K-1}$.
	
	We prove the claim by induction on $k$. For $k=1$, the statement is trivial. Suppose that $k\geq 2$ and that the claim is true for all positive integers smaller than $k$. As $v$ received color $k$, we can find $k-1$ vertices $u_{1},\dots,u_{k-1}\in C(v)$ such that the color of $u_{i}$ is $i$, for $i=1,\dots,k-1$. By the induction hypothesis, we have  $|T(u_{i})|\geq 2^{i-1}$. Since the trees $T(u_{1}),\dots,T(u_{k-1})\subset T(v)$ are pairwise disjoint, we obtain $|T(v)|\geq 1+\sum_{i=1}^{k-1}2^{i-1}=2^{k-1}$, as required.\hfill$\Box$

\bigskip

For the proof of part (ii) of Theorem \ref{thm:chi_ug}, we need the following technical lemma.

\begin{lemma}\label{lemma:bipartite}
	Let $A$ and $B$ be two $m$-element sets and let $G$ be the random graph on $A\cup B$ in which every $a\in A$ and $b\in B$ are joined by an edge independently with probability $p=\frac{d}{m}$.

Then the probability that there exist $X\subset A$ and $Y\subset B$ such that $|X||Y|\geq 3m^{2}/d$ and there is no edge between $X$ and $Y$ is at most $2^{-m}$.
\end{lemma}

\begin{proof}
	Let $N= \frac{3m^{2}}{d}$. For any $X\subset A$ and $Y\subset B$, let $I(X,Y)$ denote the event that there exists no edge between $X$ and $Y$. Obviously, we have $\mathbb{P}(I(X,Y))=(1-p)^{|X||Y|}\leq e^{-p|X||Y|}$. This yields
	$$\mathbb{P}\left(\bigcup_{\substack{X\subset A, Y\subset B\\ |X||Y|\geq N}}I(X,Y)\right)\leq \sum_{\substack{X\subset A, Y\subset B\\ |X||Y|\geq N}}e^{-p|X||Y|}\leq 2^{2m}e^{-pN}<2^{-m}.$$
	\hfill$\Box$
\end{proof}

\textit{Proof of Theorem \ref{thm:chi_ug}, part (ii).} Assume that $n\geq 2^{10r}$, and let $N=3n$, $k=\frac{\log_{2} N}{10r}$, and $m=\frac{N}{k}$. If $G$ is a graph whose vertex set is a subset of the integers, a \emph{monotone path} in $G$ is a path with vertices $c_{0}<c_{1}<\dots<c_{t}$ and edges $c_{i}c_{i+1}$ for $i=1,\dots,t-1$. A pair of vertices $\{a,b\}$ of $G$ is called \emph{bad}, if there exist two edge-disjoint monotone paths whose endpoints are $a$ and $b$.
	
Our goal is to construct a graph $G$ on the vertex set $\{1,...,N\}$ satisfying the following three conditions:
	 \begin{enumerate}
	  \item $G$ has no independent set of size larger than $7m$,
 	  \item $G$ has at most $\frac{N}{3}$ bad pairs of vertices,
	  \item the number of cycles in $G$ of length smaller than $r$ is at most $\frac{N}{3}$.	
     \end{enumerate}
Suppose that such a graph $G$ exists. Let $G'$ denote the graph obtained from $G$ by deleting $\frac{2N}{3}$ vertices: at least one vertex from every bad pair and at least one vertex from every cycle of length smaller than $r$. Then $G'$ has $n$ vertices and girth at least $r$. Condition 1 implies that the chromatic number of $G'$ is at least $\frac{n}{7m}>\frac{1}{10^{3}r}\log_{2}n$. Define a partially ordered set $P$ with partial ordering $<_{P}$ on $V(G')$ in such a way that $a<_{P} b$ if and only if $a<b$ and there exists a monotone path in $G'$ with endpoints $a$ and $b$. Then $P$ meets all the requirements of part (ii) of the theorem. Indeed, as $G'$ has no bad pair of vertices, the cover graph of $P$ is equal to $G'$, and $P$ is uniquely generated.
\smallskip

We construct a graph $G$ with the above three properties, as follows. Divide $\{1,\dots,N\}$ into $k$ intervals of size $m$, denoted by $A_{1},...,A_{k}$. For every $1\leq i<j\leq k$ and for any $x\in A_{i}$, $y\in A_{j}$, join $x$ and $y$ by an edge independently with probability $p_{ij}=\frac{2^{j-i}}{m}$. Denote the resulting graph by $G$.

First, we show that, with probability larger than $\frac{2}{3}$, condition 1 is satisfied: $G$ does not contain an independent set of size larger than $7m$.  Let $\mathcal{A}$ denote the event that for every pair $(i,j)$ with $1\leq i<j\leq k$, and for every pair of subsets $X\subset A_{i}$ and $Y\subset A_{j}$ with no edge running between $X$ and $Y$, we have $|X||Y|<3m^{2}2^{i-j}$. By Lemma \ref{lemma:bipartite}, for a fixed pair $(i,j)$ with $1\leq i<j\leq k$, with probability at least $1-2^{-m}$ there exists no $X\subset A_{i}$ and $Y\subset A_{j}$ such that $|X||Y|\geq 3m^{2}2^{i-j}$ and there is no edge between $X$ and $Y$. As there are fewer than $k^{2}$ different pairs $(i,j)$ with $1\leq i<j\leq k$,  we have $\mathbb{P}(\mathcal{A})\geq 1-k^{2}2^{-m}>\frac{2}{3}$.

We show that if $\mathcal{A}$ happens, then $G$ has no independent set of size larger than $7m$. Suppose for contradiction that $I\subset V(G)$ is an independent set with $|I|>7m$. For $i=1,...,k$, let $I_{i}=I\cap A_{i}$. Clearly, there exists an index $1\leq h\leq k$ such that $\sum_{i=1}^{h}|I_{i}|\geq 3m$ and $\sum_{i=h+1}^{k}|I_{i}|\geq 3m$. Then we have
   \begin{equation}\label{equ:ind}
   9m^{2}\leq \left(\sum_{i=1}^{h}|I_{i}|\right)\left(\sum_{i=h+1}^{k}|I_{i}|\right)=\sum_{i=1}^{h}\sum_{j=h+1}^{k}|I_{i}||I_{j}|\leq \sum_{i=1}^{h}\sum_{j=h+1}^{k}3m^{2}2^{i-j},
   \end{equation}
where the last inequality holds if $\mathcal{A}$ occurs. However,
   $$\sum_{i=1}^{h}\sum_{j=h+1}^{k}2^{i-j}\leq \sum_{l=1}^{k}l2^{-l}<2,$$
which contradicts the left-hand side of (\ref{equ:ind}).

Next, we prove that the probability that $G$ satisfies condition 2 is larger than $\frac{2}{3}$. Let $X$ stand for the number of bad pairs of vertices in $G$, and let $\mathcal{B}$ denote the event that $X \leq \frac{N}{3}$. Let $x\in A_{i}$ and $y\in A_{j}$, where $1\leq i<j\leq k$. Let $x=v_{0},v_{2},\dots,v_{l}=y$ such that $v_{t}\in A_{i_{t}}$ for $t=0,\dots,l$, where $i=i_{0}<\dots<i_{l}=j$. The probability that $v_{0},\dots,v_{l}$ is a monotone path in $G$ is
   $$\prod_{t=0}^{l-1}\frac{2^{i_{l+1}-i_{l}}}{m}=\frac{2^{j-i}}{m^l}<\frac{2^{k}}{m^{l}}.$$
There are $\binom{j-i-1}{l-1}m^{l-1}<2^{k}m^{l-1}$ ways to choose the vertices of a monotone path of length $l$ with endpoints $x$ and $y$. Hence, the probability that there exist two edge-disjoint monotone paths with endpoints $x$ and $y$, where one of these paths has length $l$ and the other has length $l'$, is at most
   $$(2^{k}m^{l-1})(2^{k}m^{l'-1})\frac{2^{k}}{m^{l}}\frac{2^{k}}{m^{l'}}=\frac{2^{4k}}{m^{2}}.$$
There are fewer than $k^{2}$ ways to choose $(l,l')$, so the probability that $\{x,y\}$ is a bad pair of vertices is less than $\frac{k^{2}2^{4k}}{m^{2}}<\frac{1}{9n}$. Therefore, we have $\mathbb{E}(X)< N^{2}\frac{1}{9N}=\frac{N}{9}$. Applying Markov's inequality, we obtain that $1-\mathbb{P}(\mathcal{B})=\mathbb{P}(X > \frac{N}{3}) < \frac{1}{3}$.

Finally, we show that $G$ satisfies condition 3, with probability larger than $\frac{2}{3}$. Let $Y$ be the number of cycles of length at most $r-1$ in $G$, and let $\mathcal{C}$ denote the event that $Y\leq \frac{N}{3}$. Let $p=n^{-(r-1)/r}$. Note that each pair of vertices in $G$ is joined by an edge with probability at most $\frac{2^{k}}{m}<p$. Then we have
   $$\mathbb{E}(Y)<\sum_{l=3}^{r-1}N^{l}p^{l}<rN^{\frac{r-1}{r}}<\frac{N}{9}.$$
Indeed, there are $\frac{(l-1)!}{2}\binom{N}{l}<N^{l}$ possible copies of the cycle of length $l$, and the probability that a fixed copy of such a cycle appears in $G$ is at most $p^{l}$. Applying Markov's inequality, we get $1-\mathbb{P}(\mathcal{C})=\mathbb{P}(Y > \frac{N}{3})<\frac{1}{3}$.
   	
In conclusion, we proved that $\mathbb{P}(\mathcal{A}),\mathbb{P}(\mathcal{B}),\mathbb{P}(\mathcal{C})>\frac{2}{3}$. Thus, the probability that the event $\mathcal{A}\wedge\mathcal{B}\wedge\mathcal{C}$ occurs is nonzero. This means that there exists a graph $G$ satisfying conditions 1,2, and 3, which completes the proof of the theorem.\hfill$\Box$

\section{Acknowledgements}
We are grateful to Bartosz Walczak for valuable discussions. He gave a direct construction proving Theorem~\ref{character} and pointed out where the components of the statement appeared in the literature.


\begin{thebibliography}{99}
\bibitem{AsG60}  Asplund, E., Gr\"unbaum, B.: On a coloring problem.  Math. Scand. {\bf 8}, 181--188 (1960)

\bibitem{B59}  Benzer, S.: On the topology of the genetic fine structure. Proceedings of the National Academy of Sciences of the United States of America {\bf 45}(11), 1607--1620 (1959)

\bibitem{Bo77} Bollob\'as, B.: Colouring lattices. Algebra Universalis {\bf 7}, 313--314 (1977)

\bibitem{DaM19}  Davies, J., McCarty, R.: Circle graphs are quadratically $\chi$-bounded.  arXiv:1905.11578

\bibitem{EET76} Ehrlich, G., Even, S., Tarjan, R. E.: Intersection graphs of curves in the plane.  J. Combin. Theory, Ser. B {\bf 21}(1), 8--20  (1976)

\bibitem{Er59}  Erd\H os, P.: Graph theory and probability. Canadian J. Math. {\bf 11}, 34--38 (1959)

\bibitem{ErH64} Erd\H os, P., Hajnal, A.: Some remarks on set theory. IX. Combinatorial problems in measure theory and set theory. Michigan Math. J. {\bf 11}(2), 107--127 (1964)

\bibitem{FoP12} Fox, J., Pach, J.: String graphs and incomparability graphs. Adv. Math. {\bf 230}(3), 1381--1401 (2012)

\bibitem{GoRU83} Golumbic, M., Rotem, D., Urrutia, J., Comparability graphs and intersection graphs.  Discrete Math. {\bf 43}, 37--46 (1983)

\bibitem{Gr78}  Graham, R. L.: Problem 1, Open Problems at 5th Hungarian Colloquium on Combinatorics (1976). in:  Combinatorics, Vol. II (A. Hajnal and V. T. S\'os, eds.), North-Holland, Amsterdam, 1195 (1978)

\bibitem{Gy85} Gy\'arf\'as, A.: On the chromatic number of multiple interval graphs and overlap graphs.  Discrete Math. {\bf 55}, 161--166 (1985)

\bibitem{KoK97} Kostochka, A., Kratochv\'{\i}l, J.: Covering and coloring polygon-circle graphs.  Discrete Math {\bf 163}(1-3), 299--305  (1997)

\bibitem{Kr91} Kratochv\'{\i}l, J.: String graphs. I. The number of critical nonstring graphs is infinite.  J. Combin. Theory Ser. B {\bf 52}(1), 53--66 (1991)

\bibitem{KrN91} K\v r\'{\i}\v z, I., Ne\v set\v ril, J.: Chromatic number of Hasse diagrams, eyebrows and dimension.  Order {\bf 8}(1), 41--48 (1991)

\bibitem{Lo83} Lov\'asz, L.: Perfect graphs. in:  Selected Topics in Graph Theory, vol. 2, Academic Press, London,  55--87 (1983)

\bibitem{MiP93}  Middendorf, M., Pfeiffer, F.: Weakly transitive orientations, Hasse diagrams and string graphs. in: Graph theory and combinatorics (Marseille-Luminy, 1990), Discrete Math. {\bf 111}(1-3), 393--400 (1993)

\bibitem{MuWW18} M\"utze, T., Walczak, B., Wiechert, V.: Realization of shift graphs as disjointness graphs of 1-intersecting curves in the plane. arXiv:1802.09969

\bibitem{NeR79} Ne\v set\v ril, J., R\"odl, V.: A short proof of the existence of highly chromatic hypergraphs without short cycles.  J. Combin. Theory Ser. B {\bf 27}(2), 225--227 (1979)

\bibitem{PaTT17} Pach, J., Tardos, G., T\'oth, G.: Disjointness graphs of segments. in: {\em 33rd International Symposium on Computational Geometry, SoCG 2017}, {\bf 77},  Leibniz Zentrum, Dagstuhl, 59:1--59:15 (2017)

\bibitem{PaT19}  Pach, J., Tomon, I.: On the chromatic number of disjointness graphs of curves. in: {\em 35th International Symposium on Computational Geometry, SoCG 2019}, {\bf 129} Leibniz Zentrum, Dagstuhl, 54:1--54:17  (2019)

\bibitem{PaT94}  Pach J., T\"or\H ocsik, J.: Some geometric applications of Dilworth's theorem.  Discrete Comput. Geom. {\bf 12}(1), 1--7 (1994)

\bibitem{PaKK14} Pawlik, A., Kozik, J., Krawczyk, T., Laso\'n, M., Micek, P, Trotter, W. T., Walczak, B.:Triangle-free intersection graphs of line segments with large chromatic number.  J. Combin. Theory Ser. B {\bf 105}, 6--10 (2014)

\bibitem{RoW18}  Rok, A., Walczak, B.: Outerstring graphs are $\chi$-bounded.  https://arxiv.org/abs/1312.1559

\bibitem{Si66} Sinden, F. W.: Topology of thin film RC-circuits.  Bell System Technical Journal {\bf 45}, 1639--1662 (1966)

\bibitem{To00}  T\'oth, G.: Note on geometric graphs.  J. Combin. Theory Ser. A {\bf 89}(1), 126--132 (2000)

\bibitem{V95}  Vogt, H. G.: Le{\c c}ons sur la r\'esolution alg\'ebrique des \'equations. Nony, p. 91 (1895)



\end{thebibliography}

\end{document}